\documentclass[reqno]{amsart}
\usepackage{amssymb}
\usepackage{mathrsfs}
\usepackage{cite}
\usepackage{tikz-cd}
\usepackage{MnSymbol}
\usepackage{a4wide}
\usepackage{todonotes}
 
\usepackage{xcolor}

\DeclareMathOperator\C{\mathbb C}
\DeclareMathOperator\R{\mathbb R}
\DeclareMathOperator\Z{\mathbb Z}

\newtheorem{theorem}{Theorem}[section]
\newtheorem{lemma}[theorem]{Lemma}
\newtheorem{cor}[theorem]{Corollary}
\newtheorem{prop}[theorem]{Proposition}
\theoremstyle{definition}
\newtheorem{definition}[theorem]{Definition}
\newtheorem{example}[theorem]{Example}

\theoremstyle{remark}
\newtheorem{remark}[theorem]{Remark}

\newcommand{\dontprint}[1]\relax

\newcommand{\Ga}{\Gamma}

\newcommand{\BB}{{\mathcal B}}
\newcommand{\TT}{{\mathcal T}}

\newcommand{\MM}{{\mathcal M}}

\newcommand{\de}{\delta}
\newcommand{\sub}{\subset}

\newcommand{\ov}{\overline}

\newcommand{\om}{\omega}

\renewcommand{\a}{\alpha}
\renewcommand{\b}{\beta}

\newcommand{\coker}{\operatorname{coker}}

\newcommand{\we}{\wedge}
\newcommand{\De}{\Delta}
\newcommand{\rk}{\operatorname{rk}}

\renewcommand{\Re}{\operatorname{Re}}

\newcommand{\cT}{\mathcal{T}}

\numberwithin{equation}{section}
\title{Convergence of integrals on the moduli spaces of curves and cographical matroids}
\author{Alexander Polishchuk}
\author{Nicholas Proudfoot}
\thanks{A.P. is partially supported by the NSF grants DMS-2001224, DMS-2349388, by the Simons Travel grant MPS-TSM-00002745, 
and within the framework of the HSE University Basic Research Program. N.P. is partially supported by the NSF grants DMS-1954050, DMS-2053243, and DMS-2344861.}
\begin{document}

\begin{abstract}
We determine the convergence regions of certain local integrals on the moduli spaces of curves in neighborhoods of fixed stable curves
in terms of the combinatorics of the corresponding graphs. 
\end{abstract}

\maketitle

\section{Introduction}

Let $\ov{\MM}_g$ denote the moduli space of stable curves of genus $g$, and let $\De^{ns}\sub \ov{\MM}_g$ denote the non-separating node boundary divisor.

We consider 
a stratum $\MM_{\Ga}\sub \ov{\MM}_g$, consisting of curves of given combinatorial type.
Here $\Ga$ is a stable (multi-)graph of genus $g$ (possibly with multiple edges and loops, with marking by genus on vertices).
Let $\varphi$ be a top holomorphic form on $\ov{\MM}_g$ defined in a neighborhood in $\ov{\MM}_g$ of a point $C_0\in \MM_{\Ga}$, with a pole of order
$1$ along $\De^{ns}$. We are interested in the region of convergence in $s\in\C$ of the integral
$$I_{\Ga}(\varphi,s)=\int_{B_{\Ga}}\frac{\varphi\we\ov{\varphi}}{|\det(\tau-\ov{\tau})|^s},$$
where $B_{\Ga}$ is a sufficiently small ball in an \'etale neighborhood of $C_0$ in $\ov{\MM}_g$, and $\tau$ is the period matrix. Integrals of this type (for $s=5$) appear in calculation
of vacuum amplitudes in superstring theory after integrating out the odd variables and using the GSO projection (see \cite{Witten}).


The point is that $\det(\tau-\ov{\tau})$ has a logarithmic growth near $\De^{ns}$ that offsets the poles of $\varphi\we \ov{\varphi}$ for sufficiently large $s$.
The precise region of convergence of $I_{\Ga}(\varphi,s)$ depends on a graph $\Ga$.

For example, if $\Ga$ has a single vertex of genus $g-1$ with a loop, i.e., we are at the generic point of $\De^{ns}$, then $\det(\tau-\ov{\tau})=u\cdot \log |t|$, where
$u$ is invertible and $t=0$ is a local equation of $\De^{ns}$. So the convergence of our integral for such $\Ga$ is the same as for $\frac{dt\we d\ov{t}}{t\cdot \ov{t}\cdot (\log |t|)^s}$.
Thus, in this case the integral absolutely converges for $\Re(s)>1$ and diverges for $s=1$.

Our main result, Theorem A below, determines the convergence threshold for each $\Ga$. In particular, it shows that for each genus $\ge 6$, there exists a stable graph $\Ga$
such that $I_{\Ga}(\varphi,5)$ diverges. This means that the definition of superstring vacuum amplitudes for $g\ge 6$ requires some additional regularization procedure at the non-separating node 
boundary divisor in addition to the GSO projection.

Using the asymptotics of $\tau$ near $C_0$ (controled by the monodromy around branches of $\De^{ns}$), in the case when $C_0$ has rational components, we relate the integral $I_{\Ga}$ to another integral defined 
in terms of the graph $\Ga$.

Let $E(\Ga)$ denote the set of edges of a connected graph $\Ga$. We introduce independent variables $x_e$ associated $e\in E(\Ga)$.
Let $b=b(\Ga)$ denote the 1st Betti number of $\Ga$ (which is equal to $g$, the arithmetic genus of $C_0$),
and let $c_1,\ldots,c_b$ be a collection of simple cycles in $\Ga$ giving a basis in $H_1(\Ga)$. For each edge $e\in E(\Ga)$, we define the index
$$(c_i,c_j)_e=\begin{cases} 1, &c_i \text{ and } c_j \text{ pass through } e \text{ in the same direction},\\
-1, &c_i \text{ and } c_j \text{ pass through } e \text{ in the opposite directions},\\
0, &\text{ otherwise},\end{cases}$$
and consider the $b\times b$-matrix $A=(a_{ij})$ with
\begin{equation}\label{a-matrix-eq}
a_{ij}=\sum_{e\in E(\Ga)}(c_i,c_j)_e\cdot x_e.
\end{equation}
Let $\psi_\Ga\in \Z[x_e]_{e\in E(\Ga)}$ denote the Kirchoff polynomial (aka the first Symanzik polynomial) of $\Ga$ (see \cite[Sec.\ 2]{Bloch}),
defined as the determinant 
$$\psi_\Ga=\det(a_{ij}).$$ 
This polynomial also has an expansion
\begin{equation}\label{psi-expansion}
\psi_\Ga=\sum_T \prod_{e\not\in T} x_e,
\end{equation}
where $T$ runs over all spanning trees of $\Ga$ (see \cite[Prop.\ 3.4]{Bloch} and Prop.\ \ref{Kirch-prop} below).
Note that $\psi_\Ga$ is homogeneous of degree $b(\Ga)$.


Let $E'(\Ga)\sub E(\Ga)$ denote the set of edges which are not bridges. We consider the integral
$$J_{\Ga}(s):=\int_B \frac{\prod_{e\in E(\Ga)} dz_ed\ov{z}_e}{|\psi_\Ga(\ln |z_\bullet|)|^s\cdot \prod_{e\in E'(\Ga)} z_e\ov{z}_e},$$
where $B$ is a small ball around $0$ in $\C^{E(\Ga)}$, $(z_e)$ are complex coordinates on $\C^{E(\Ga)}$.

Given a connected graph $\Ga$ without bridges (not necessarily stable), we define a rational constant $c(\Ga)>0$ as follows.
Consider the vector space $\R^{E(\Ga)}$ with the basis corresponding to edges of $\Ga$, and let $v\in \R^{E(\Ga)}$ denote
the sum of all basis vectors. For each spanning tree $T$, consider
the vector
$$v_T:=\sum_{e\not\in T} e\in \R^{E(\Ga)}.$$
Now we set 
$$c(\Ga)=\inf \left\{\sum c_T \ \big{|} \ \sum_T c_Tv_T\geq v, c_T\in \R_{\ge 0}\right\},$$
where $w \geq v$ means that $w-v\in \R_{\ge 0}^{E(\Ga)}$.

It is easy to see that 
\begin{equation}\label{c-Gamma-ineq}
c(\Ga)\ge \frac{e(\Ga)}{b(\Ga)},
\end{equation}
where $e(\Ga)=|E(\Ga)|$ is the number of edges and $b(\Ga)$ is the 1st Betti number of $\Ga$. Indeed, let $\varphi:\R^{E(\Ga)}\to \R$
denote the map given by the sum of all coordinates. Then $\varphi(v)=e(\Ga)$, $\varphi(v_T)=b(\Ga)$, so the inequality $\sum_T c_Tv_T\geq v$
implies $b(\Ga)\cdot (\sum_T c_T)\ge e(\Ga)$.

For an arbitrary connected $\Ga$ (possibly with bridges), we set $c(\Ga):=c(\Ga')$, where $\Ga'$ is obtained from $\Ga$ by contracting all the bridges.

\bigskip

\noindent
{\bf Theorem A}. {\it The integral
 $J_{\Ga}(s)$ converges for $\Re(s)>c(\Ga)$ (for a sufficiently small ball around $0$) and diverges for $s=c(\Ga)$ (and any ball around $0$). 
 If $\Ga$ is stable and all components of $C_0$ are rational, the same assertions hold for the integral
$I_{\Ga}(\varphi,s)$.}
 
\bigskip

A simple example is the $n$-gon graph $\Ga=P_n$. It is easy to see that one has $c(P_n)=n$. Indeed, the vectors $v_T$ are just the basis vectors $e_i$
and the condition $\sum c_ie_i\ge v$ means that $c_i\ge 1$, so the minimal $\sum c_i$ is $n$. 
There are similar stable graphs with all vertices of genus $0$. For example, if $\Ga$ is a $2n$-gon with every other side doubled then $c(\Ga)=n$. Note that the genus of this graph is $n+1$.
This shows that the boundary $\Re(s)=c$ of the convergence halfplane for $I_{\Ga}(\varphi,s)$ can have arbitrary large $c$ as genus grows.


The proof of Theorem A uses some combinatorics of the {\it cographical matroid} associated to $\Ga$ to reduce to the case of graphs $\Ga$ for which inequality \eqref{c-Gamma-ineq}
becomes an equality (recall that the bases of the cographical matroid are complements to spanning trees, see \cite[Sec.\ 2.3]{Ox}). The key combinatorial result needed for Theorem A is that starting from any graph
$\Ga$, one can contract some edges in $\Ga$ so that the resulting graph $\ov{\Ga}$ satisfies $c(\ov{\Ga})=e(\ov{\Ga})/b(\ov{\Ga})$ (see Cor.\ \ref{main-cor}).

Note that the integral $J_{\Ga}(s)$ is similar to the Euler-Mellin integrals considered in \cite{BFP} (but with a different integration domain) and our convergence result is similar to  
\cite[Thm.\ 2.2]{BFP} (see Remark \ref{BFP-rem}).

\bigskip

{\it Convention}. By a graph we mean a connected undirected multigraph, possibly with multiple edges and loops.

\medskip

{\it Acknowledgment}. We thank Erik Panzer for useful discussions.

\section{Kirchoff polynomial and the asymptotics of the period matrix}

\subsection{Kirchoff polynomial}

We are going to give a proof of the determinant identity relating two definitions of $\psi_\Ga$,
using Cauchy-Binet theorem, and also get a bit more information about the corresponding matrix. There is also a simple recursive proof of this determinant formula
in \cite{Bloch}.

Let us consider a more general setup where we are given a surjective morphism of free $\Z$-modules of finite rank
$$\a:\Z^E\to H,$$
with the property that for every subset $S\sub E$, the cokernel $\coker(\a_S)$ is a free $\Z$-module, where $\a_S:\Z^S\to H$ is the restriction of $A$.
\footnote{Such a morphism is known as a unimodular collection of vectors in $H$.}

We will apply it in the case when $E=E(\Ga)$, $H=H^1(\Ga)$, and $\a$ is the natural projection. Note that in this case $\coker(\a_S)$ is canonically identified with
$H^1(\Ga^S)$, where $\Ga^S$ is obtained from $\Ga$ by deleting edges in $S$.

With $\a$ as above we associate a symmetric bilinear form on $H^*$ with entries which are linear forms in independent variables $(x_e)_{e\in E}$:
$$B(\xi_1,\xi_2)=\sum_{e\in E} \xi_1(\a(e))\cdot \xi_2(\a(e))\cdot x_e.$$
Thus, we have a well defined discriminant $\det(B)\in R:=\Z[x_e \ |\ e\in E]$.

Let us say that $S\sub E$ is a {\it basis} if $\a_S:\Z^S\to H$ is an isomorphism. We denote by $\BB$ the set of all bases.


\begin{prop}\label{Kirch-prop}
(i) One has $\det(B)=\sum_{S\in \BB} \prod_{e\in S} x_e$.

\noindent
(ii) Let us view $B$ as a nondegenerate form over the field $QR$, the fraction field over $R$, and let $B^{-1}$ denote the corresponding bilinear form on the dual space
$H^\vee$ with values in $QR$. Then setting $x_e=\ln(y_e)$, where $y_e>0$, we have
$$\lim_{y\to 0} B^{-1}(\ln(y_e))=0.$$ 
\end{prop}

\begin{proof}
(i) Let $f_1,\ldots,f_h$ be a $\Z$-basis of $H$, so we can view $\a$ as a matrix $E\times h$-matrix. Then the matrix of $B$ is given by
$$B=\a\cdot D(x)\cdot \a^t,$$
where $D(x)$ is the diagonal $E\times E$-matrix with the entries $(x_e)$.
We can calculate $\det(B)$ by applying the Cauchy-Binet's theorem to the decomposition of $B$ in $\a$ and $D(x)\a^t$.
In this theorem we need to sum over subsets $S\sub E$, where $|S|=h$, such that the corresponding minor of $\a$ is nonzero.
This condition is equivalent to the condition that $\a_S:\Z^S\to H$ has nonzero determinant. Since $\coker(\a_S)$ is free by assumption,
this is equivalent to $S$ being a basis, in which case the corresponding minor is $\pm$. Since the corresponding minor of $\a^t$ is the same sign
times $\prod_{e\in S} x_e$, the assertion follows.

\noindent
(ii) In the dual basis $(f_i^*)$ the matrix of $B^{-1}$ is the inverse matrix of $B$. Hence, it is enough to prove that every $(h-1)\times (h-1)$-minor $M$ of $B$
satisfies
$$\lim_{y\to 0} \frac{M(\ln(y))}{\det B(\ln(y))}=0.$$
Due to the formula for the determinant, it is enough to prove that every monomial appearing in $M$ is a constant multiple of $\prod_{e\in S} x_e$, for some
$S\sub S'$ with $|S|=h-1$ and $S'$ a basis. Indeed, without loss of generality we can assume that $M$ corresponds to the rows $1,\ldots,h-1$.
Then by the Cauchy-Binet's formula, the monomials appearing in $M$ would correspond to some subsets $S\sub E$ with $|S|=h-1$ such that
the map $\Z^S\to H/\Z\cdot f_h$ induced by $\a_S$ is nondegenerate. But this implies that $\a_S$ is injective, and so the image of $\a_S$ is of rank $h-1$.
Since $\a$ is surjective, there exists an element $s\in E\setminus S$ such that $\a(s)$ is not contained in the image of $\a_S$, hence for $S'=S\cup\{s\}$,
the map $\a_{S'}$ is nongenerate, i.e., $S'$ is a basis.
\end{proof}

We are interested in the case of the natural projection $\a:\Z^{E(\Ga)}\to H^1(\Ga)$ for a graph $\Ga$. In this case, $\BB$ is the set of bases of the cographical matroid
associated with $\Gamma$.
Choosing a basis of simple cycles $(c_i)$,
we can view the rows of $\a$ as coefficients of the edges in $c_i$ (with respect to a fixed orientation of all edges). Then the matrix of the symmetric form $B$ will be exactly
the matrix \eqref{a-matrix-eq}. Hence, $\det(B)$ gets identified with the polynomial $\psi_{\Ga}$ and we derive the expansion \eqref{psi-expansion}
Note that in this case bases are exactly complements
to spanning trees. In particular, if an edge $e$ is a bridge then it is not contained in any bases, so $\psi_\Ga$ depends only on variables corresponding to 
 {\it non-separating edges} (i.e., edges that are not bridges).

\subsection{Asymptotics of the period matrix}

Assume that $\Ga$ is a stable graph of genus $g$, with all vertices of genus $0$, and let $C_0$ be a stable curve with rational components with the dual graph $\Ga$.
It is well known that the arithmetic genus of $C_0$ is $g$, so we can view $C_0$ as a point of $\ov{\MM}_g$. Then the set of edges $E(\Ga)$ is in bijection
with the branches of the boundary divisor through $C_0$, and the subset $E^{ns}(\Ga)\sub E(\Ga)$ of non-separating edges is
in bijection with the branches of $\De^{ns}$ through $C_0$.

Note that if the graph $\Ga$ is trivalent then the corresponding stratum $\MM_{\Ga}$ is a point and $e(\Ga)=3g-3$. Otherwise, the stratum $\MM_{\Ga}$ has positive dimension.

For every non-separating edge $e$, we have the corresponding vanishing cycle $\a_e\in H_1(C)$, where $C$ is a smooth curve close to $C_0$,
and the corresponding monodromy transformation $M_e$ on $H_1(C)$ has form 
$$M_e(x)=(x\cdot \a_e)\a_e+x.$$
More precisely, $C$ is obtained from the collection of spheres numbered by the set $V(\Ga)$ of vertices of $\Ga$ by connecting them with tubes numbered by the set
of edges $E(\Ga)$. We fix an orientation of $\Ga$ and let $\b_e$ denote a path along the tube corresponding to $e$ going in the direction of $e$.
Then we define $\a_e$ as the class of a circle around the tube corresponding to $e$, 
so that $(\b_e\cdot\a_e)=1$.


The subgroup $A\sub H_1(C)$ generated by $(\a_e)$ is maximal isotropic, and we have natural identifications
$$H_1(C)/A\simeq H_1(C_0)\simeq H_1(\Ga)$$
(see \cite[Sec.\ 6]{Bloch}).

Let $(\a_1,\ldots,\a_g,\b_1,\ldots,\b_g)$ be a standard symplectic basis of $H_1(C)$ (so $(\b_j\cdot\a_i)=\de_{ij}$) such that $A=\Z\a_1+\ldots+\Z\a_g$.
In particular, $M_e$ does not change $\a_i$'s. Let $\om_1,\ldots,\om_g$ be the basis of $H^0(C,\om_C)$, normalized by $\int_{\a_i}\om_j=\de_{ij}$.
Then $M_e$ preserves $(\om_j)$, and changes the periods $\tau_{ij}=\int_{\b_i}\om_j$ to
$$M_e:\tau_{ij}\mapsto \int_{M_e(\b_i)}\om_j=\tau_{ij}+(\b_i\cdot \a_e)\cdot (\a_e)_j=\tau_{ij}+(\b_i\cdot \a_e)\cdot (\b_j\cdot \a_e),$$
where the integers $(\a_e)_j$ are determined from $\a_e=\sum_j (\a_e)_j\cdot \a_j$.

Let $z_e$ denote a local equation of the branch of the boundary divisor near $C_0$ corresponding to the edge $e$. Then the monodromy $M_e$
acts on $\ln(z_e)/(2\pi i)$ as $\ln(z_e)/(2\pi i)\mapsto \ln(z_e)/(2\pi i)+1$.
Hence,
$$\tau'_{ij}:=\tau_{ij}-\sum_{e\in E}\frac{\ln(z_e)}{2\pi i}\cdot (\b_i\cdot \a_e)\cdot (\b_j\cdot \a_e)$$
are invariant under all monodromy transformations $M_e$. In other words, we have
\begin{equation}\label{tau-tau'-eq}
\tau=\tau_0(\frac{\ln(z)}{2\pi i})+\tau',
\end{equation}
where $\tau=(\tau_{ij})$, $\tau'=(\tau'_{ij})$, and $\tau_0(x)$ is the matrix with coefficients in $\Z[x_e \ |\ e\in E]$ with the entries
$$\tau_0(x)_{ij}=\sum_{e\in E} x_e\cdot (\b_i\cdot \a_e)\cdot (\b_j\cdot \a_e).$$

In fact, it follows from the nilpotent orbit theorem (see \cite[Sec.\ 9]{Bloch}) that in Eq.\ \eqref{tau-tau'-eq} the term $\tau'$ is regular near $C_0$.
We will use this to compute the asymptotics for $\det(\tau-\ov{\tau})$ near $C_0$.

\begin{lemma}\label{asympt-lem}
Near $C_0$ one has 
$$\det(\tau-\ov{\tau})=\psi_{\Ga}(\frac{\ln(|z|)}{\pi i})\cdot (1+f),$$
where $z=(z_e)$ and $f\to 0$ as $C\to C_0$.
\end{lemma}

\begin{proof}
Let us choose a symplectic basis of $H_1(C)$ as follows. First, let $(c_i)_{i=1,\ldots,g}$ be a basis of simple cycles in $H_1(\Ga)$, and let $(\a_i)_{i=1,\ldots,g}$
be the dual basis of $A$ with respect to the pairing between $A$ and $H_1(C)/A\simeq H_1(\Ga)$ induced by the intersection pairing.
Let $(\b_i)_{i=1,\ldots,g}$ be a set of mutually orthogonal classes in $H_1(C)$ projecting to $(c_i)$ under the projection $H_1(C)\to H_1(\Ga)$ (such a set exists since the intersection pairing is perfect).
Then $(\a_i,\b_i)$ is a standard symplectic basis. Furthermore,
due to our definition of $\a_e$, the intersection index $(\b_i\cdot \a_e)$ is $1$ (resp., $-1$) exactly when $c_i$ passes through $e$ in the direction of the orientation (resp., in the opposite direction).

It follows that $\tau_0(x)$ coincides with the matrix $A=(a_{ij})$ given by \eqref{a-matrix-eq}.
Since $\tau_0(x)$ depends on $(x_e)$ linearly with integer coefficients, \eqref{tau-tau'-eq} gives
$$\tau-\ov{\tau}=A(\frac{\ln |z|}{\pi i})+\tau'-\ov{\tau'}=A(\frac{\ln |z|}{\pi i})\cdot \bigl(1+A(\frac{\ln |z|}{\pi i})^{-1}\cdot (\tau'-\ov{\tau'})\bigr),$$
where $\tau'-\ov{\tau'}$ is regular near $C_0$. Now by Proposition \ref{Kirch-prop}(ii), we have
$$\lim_{C\to C_0}A(\frac{\ln |z|}{2\pi i})^{-1}\cdot (\tau'-\ov{\tau'})=0,$$
and the assertion follows.
\end{proof}

\section{Convergence region}

\subsection{Elementary observations}

\begin{lemma}\label{two-int-lem}
Let $P(x_1,\ldots,x_n)$ be a (nonzero) polynomial with non-negative coefficients. Then for $s\in \C$, 
the integral
$$\int_B \frac{\prod_{i=1}^n dz_id\ov{z}_i}{|P(-\ln |z_1|,\ldots,-\ln |z_n|)|^s\cdot \prod_{i=1}^n z_i\ov{z}_i}$$
over a sufficiently small ball $B$ around the origin in $\C^n$ converges if and only if the integral
$$\int_{[C,+\infty)^n} \frac{dx_1\ldots dx_n}{P(x_1,\ldots,x_n)^s}$$
converges for sufficiently large $C$.
\end{lemma}

\begin{proof}
This follows immediately from the change of variables $z_i=e^{-x_i+ i\phi_i}$, with $x_i>C$ and $\phi_i\in (0,2\pi)$.
\end{proof}

\begin{lemma}\label{elem-lem}
Let $C>0$. The integral
$$\int_{[C,+\infty)^n}\frac{dx_1\ldots dx_n}{(x_1+\ldots+x_n)^s}$$
converges for $\Re(s)>n$ and diverges for $s=n$.
\end{lemma}

\begin{proof} Convergence for $\Re(s)>n$ follows from the inequality $(x_1+\ldots+x_n)^n\ge x_1\ldots x_n$. 

To prove the divergence for $s=n$, we will use induction on $n$. The base case $n=1$ is well known. Now let $n>1$. 
It is enough to prove the divergence of the integral
$$\int_{x_1>C_1}\ldots \int_{x_n>C_n} \frac{dx_1\ldots dx_n}{(x_1+\ldots+x_n)^n},$$
for any large $C_1,\ldots,C_n$. Performing the integration in $x_n$ we get
$$\frac{1}{n-1}\int_{x_1>C_1}\ldots \int_{x_{n-1}>C_{n-1}}\frac{dx_1\ldots dx_{n-1}}{(x_1+\ldots+x_{n-1}+C_n)^{n-1}}.$$
It remains to use the change of variables $x_1\mapsto x_1+C_n$ and use the induction assumption.
\end{proof}

For a vector $w=(w_1,\ldots,w_n)\in \Z_{\ge 0}^n$, we denote $x^w:=x_1^{w_1}\ldots x_n^{w_n}$.

\begin{lemma}\label{easy-div-lem}
(i) Let $C>0$. Suppose for a collection of vectors $v_1,\ldots,v_N\in \Z_{\ge 0}^n$ and scalars $c_1,\ldots,c_N\in \R_{\ge 0}$ one has
$$\sum c_iv_i \ge v:=(1,\ldots,1)\in \R^n.$$
Then the integral
\begin{equation}\label{v1-vN-s-int}
\int_{[C,+\infty)^n}\frac{dx_1\ldots dx_n}{(x^{v_1}+\ldots+x^{v_N})^s}
\end{equation}
converges for $\Re(s)>c_1+\ldots+c_N$.
In particular, for $\Ga$ without bridges, the integral $J_{\Ga}(s)$ converges for $s>c(\Ga)$.

\noindent
(ii) The integral $J_{\Ga}(s)$ diverges for $s=e(\Ga)/b(\Ga)$.
\end{lemma}

\begin{proof}
(i) 
The first assertion follows from the inequality
$$(x^{v_1}+\ldots+x^{v_N})^{c_1+\ldots+c_N}\ge (x^{v_1})^{c_1}\ldots (x^{v_N})^{c_N}\ge x^v=x_1\ldots x_n$$
for $x_i\ge 1$.
The assertion about $J_{\Ga}$ follows from this using Lemma \ref{two-int-lem}.

\noindent
(ii) Set $n=e(\Ga)$, $b=b(\Ga)$.
Since $\psi_{\Ga}$ has degree $b$, we have 
$$\psi_{\Ga}(x_1,\ldots,x_n)\le (x_1+\ldots+x_n)^b.$$
Now the divergence for $s=n/b$ follows immediately from Lemmas \ref{two-int-lem} and
\ref{elem-lem}.
\end{proof}

\begin{remark}\label{BFP-rem}
The convergence statement of Lemma \ref{easy-div-lem}(i) is similar to the convergence statement \cite[Thm.\ 2.2]{BFP} about more general Euler-Mellin integrals.
Note however that our domain of integration is $[C,+\infty)^n$, where $C>0$, so our condition on $s$ is different from the one in \cite[Thm.\ 2.2]{BFP} where the integration is over $(0,+\infty)^n$. 
The converse of Lemma \ref{easy-div-lem}(i) is false, i.e., the integral \eqref{v1-vN-s-int} can converge even when there exist no $c_i\in \R_{\ge 0}$ with 
$\sum c_i v_i\ge v$ and $\Re(s)>\sum c_i$. For example, consider the integral
$$I(s)=\int_{[C,+\infty)}\frac{dx_1dx_2}{(x_1x_2^2+x_1^4x_2^3)^s}.$$
Then conditions of Lemma \ref{easy-div-lem}(i) hold only for $\Re(s)>2/5$, however, we claim that $I(s)$ converges for $\Re(s)>1/3$. Indeed, changing the variables by $x_i=y_i+C$,
we get
$$I(s)=\int_{[0,+\infty)}\frac{dy_1dy_2}{P(y_1,y_2)^s},$$
where $P(y_1,y_2)=(C+y_1)(C+y_2)^2+(C+y_1)^4(C+y_2)^3$. Since the Newton polytope of $P$ is the rectangle with the vertices $(0,0)$, $(4,0)$, $(0,3)$ and $(4,3)$,
convergence for $\Re(s)>1/3$ follows from  \cite[Thm.\ 2.2]{BFP}.
\end{remark}

Recall that we have an inequality $c(\Ga)\ge e(\Ga)/b(\Ga)$ (see \eqref{c-Gamma-ineq}).

\begin{definition} We say that $\Ga$ {\it optimal} if $c(\Ga)=e(\Ga)/b(\Ga)$.
\end{definition}

Lemma \ref{easy-div-lem} proves our assertion about convergence/divergence of $J_{\Ga}(s)$ in the case of optimal $\Ga$.
Below we will reduce the case of a general $\Ga$ to that of optimal $\Ga$.

\subsection{Combinatorial statement}

Let $M$ be a loopless matroid on the ground set $E$.  For all $S\subset E$, consider the linear functional
$$\varphi_S:\R^E\to\R$$
taking an element to the sum of its coordinates in $S$.  The {\bf base polytope} $P(M)$ consists of the vectors $w\in \R^E$
such that $0\leq \varphi_S(w) \leq \rk S$ for all $S\subset E$ and $\varphi_E(w) = \rk E$.\\

Let 
$$m=m(M):= \max\{|S| / \rk S \mid S \neq \emptyset\},$$ 
and let $\cT_0$ be the (nonempty) collection of subsets of $E$ that attain this maximum.

\begin{lemma}\label{max}
If $S,T\in\cT_0$, then $S\cup T\in \cT_0$.  In other words, $\cT_0$ has a maximal element.
\end{lemma}

\begin{proof}
For any set $U\subset E$ of nonzero rank, we have $|U| \leq m \rk U$.  Applying this inequality to $U = S \cap T$, we find that
\begin{eqnarray*}|S\cup T| &=& |S| + |T| - |S\cap T|\\ &\geq& m \rk S + m\rk T - m \rk S\cap T\\
&=& m (\rk S + \rk T - \rk S\cap T)\\
&\geq& m \rk S\cup T.
\end{eqnarray*}
Applying it next to $U = S\cup T$, we find that $|S\cup T| = m \rk S\cup T$, thus $S\cup T \in \cT_0$.
\end{proof}

Let 
$$c=c(M) := \min\{t \mid \text{there exists $w\in t P(M)$ with $w_e \geq 1$ for all $e\in E$}\},$$
and let $w$ be an element of $c P(M)$ with $w_e \geq 1$ for all $e\in E$ (a witness for $c$).

\begin{prop}\label{cm}
We have $c=m$.
\end{prop}

\begin{proof}
Let $T_0$ be the maximal element of $\cT_0$ (which exists by Lemma \ref{max}).  Since $w_e\geq 1$ for all $e\in E$ and $w\in c P(M)$, we have $|T_0| \leq \varphi_{T_0}(w) \leq c \rk T_0$, and therefore $c \geq |T_0| / \rk T_0 = m$.

Now we must prove the opposite inequality.  
We will do it by constructing an element $w\in m P(M)$ with $w_e \geq 1$ for all $e\in E$.  This construction will proceed in stages.

First let $w(0) = (1,\ldots,1)$.  We then have $\varphi_S(w(0)) = |S| \leq m \rk S$ for all $S\subset E$, 
with equality if and only if $S\in \cT_0$.  In particular, we 
do not necessarily have $\varphi_E(w(0)) = m \rk E$.  If $T_0 = E$, then $\varphi_E(w(0)) = m \rk E$, so $w(0)\in m P(M)$ and we are done.
If not, choose $e_0\in E\setminus T_0$.  Then we have $e_0\notin S$ for all $S\in \cT_0$.  That means that there exists $\epsilon > 0$
such that $\varphi_S(w(0) + \epsilon x_e) \leq m \rk S$ for all $S\subset E$.  Choose the largest such $\epsilon$, and let
$w(1) = w(0) + \epsilon x_{e_0}$.

Now let $\cT_1$ be the collection of subsets $S\subset E$ with the property that $\varphi_S(w(1)) = m \rk S$.
By an argument identical to that of Lemma \ref{max}, there is a maximal element $T_1$ of $\cT_1$.  If $T_1 = E$,
then $w(1) \in m P(M)$ and we are done.  If not, choose $e_1 \in E \setminus T_1$, and repeat the procedure to produce a new vector $w(2)$.

At some point we will have $T_k = E$, and this process will terminate with $w(k) \in m P(M)$ and $w(k)_e \geq 1$ for all $e\in E$.
\end{proof}


\begin{prop}\label{delete}
Given a loopless matroid $M$ and any $T\in \TT_0$, one has
$$c(M) = c(M|_T) = |T| / \rk T.$$
In particular, this is true for the maximal element $T_0\in \cT_0$.
\end{prop}

\begin{proof}
Since the rank function on $M|_T$ is the same as that on $M$,
it is clear that 
$$m(M) = m(M|_T) = |T| / \rk T.$$
It remains to apply Proposition \ref{cm}.
\end{proof}

\begin{cor}\label{main-cor}
For any connected graph $\Ga$ without bridges, there exists a graph $\ov{\Ga}$ obtained by contracting some edges in $\Ga$, such that
$$c(\Ga)=c(\ov{\Ga})=\frac{e(\ov{\Ga})}{b(\ov{\Ga})}$$
(the second equality means that $\ov{\Ga}$ is optimal).
\end{cor}

\begin{proof}
Let $M$ denote the cographical matroid associated with $\Gamma$.
Then $c(M)=c(\Gamma)$.  
The property that $\Gamma$ is optimal is the property that $c(M) = |E| / \rk E$.  Corollary \ref{delete} says that we can find an optimal deletion of $M$
with the same constant $c$.  That means that we can find an optimal contraction of $\Gamma$ with the same constant $c$.
\end{proof}

\begin{example}
Let $\Ga$ be an $2n$-gon with every other side doubled, so $e(\Ga)=3n$.
It is easy to see that $c(\Ga)=n$. We can collapse all the doubled sides to get the $n$-gon $\ov{\Ga}$, which is optimal and has
$c(\ov{\Ga})=n$. 
\end{example}

\subsection{Proof of Theorem A}

Lemma \ref{asympt-lem} shows that the case of the integral $I_{\Ga}(\varphi,s)$ (for stable $\Ga$ and all components of $C_0$ rational)
reduces to the case of $J_{\Ga}(s)$. Also, since $\psi_{\Ga}=\psi_{\Ga'}$, where $\Ga'$ is obtained by contracting all the bridges, it is enough to consider the
case when $\Ga$ has no bridges.

Due to Lemma \ref{easy-div-lem}(i), it remains to prove that $J_\Ga(s)$ diverges for $s=c(\Ga)$. 
Let $S\sub E(\Ga)$ denote the set of edges that get contracted to get $\ov{\Ga}$.
By Fubini theorem, it is enough to prove that for any fixed values $z_i=c_i$ with $i\in S$, the integral 
$$\int_{B'} \frac{\prod_{e\not\in S} dz_ed\ov{z}_e}{\psi_\Ga(\ln |z_\bullet|)|_{z_i=c_i, i\in S}^{c(\Ga)}\cdot \prod_{e\not\in S} z_e\ov{z}_e},$$
where $B'$ is a small ball around the origin in $\C^{E(\Ga)\setminus S}$, diverges.
Thus, by Lemma \ref{easy-div-lem}(ii) applied to $\ov{\Ga}$, it is enough to prove an inequality 
$$\psi_\Ga(x_\bullet)|_{x_e=a_e, e\in S}\le C\cdot \psi_{\ov{\Ga}}(x_\bullet),$$
for $a_e>0$, with some constant $C>0$ depending on $(a_e)$. Furthermore, it is enough that this inequality
holds for $x_e>C$.
Now we observe that for every spanning forest $T\sub \Ga$, the intersection $\ov{T}:=T\cap \ov{\Ga}$ is also a spanning forest.
This implies that for every monomial $x_T=\prod_{\not\in T} x_e$ of $\psi_{\Ga}$, one has 
$$x_T|_{x_e=a_e, e\in S}=\prod_{e\in S} a_e\cdot x_{\ov{T}}.$$
This easily leads to the claimed inequality.


\begin{thebibliography}{9}
\bibitem{BFP} C.~Berkesch, J.~Forsg{\aa}rd, M.~Passare, {\it Euler-Mellin integrals and  A -hypergeometric functions},
Michigan Math. J. 63 (2014), no. 1, 101--123.
\bibitem{Bloch} S.~Bloch, {\it Feynman Amplitudes in Mathematics and Physics}, in 
{\it Amplitudes, Hodge theory and ramification---from periods and motives to Feynman amplitudes}, 1--34, AMS, Providence RI, 2020.
\bibitem{Ox} J.~Oxley, {\it Matroid Theory}, Oxford Univ. Press, Oxford, 2011.
\bibitem{Witten} E. Witten,
  {\it Notes on Holomorphic String And Superstring Theory Measures Of Low Genus},
  Analysis, complex geometry, and mathematical physics: in honor of Duong H. Phong, 307--359,
Contemp. Math., 644, Amer. Math. Soc., Providence, RI, 2015.
\end{thebibliography}
\end{document}